\UseRawInputEncoding
\documentclass[11pt,a4paper]{article}
\usepackage{}
\usepackage{bbding}
\usepackage{amsfonts}
\usepackage{amssymb}
\usepackage{mathrsfs}
\usepackage{amsmath}
\usepackage{epsf,epsfig,amsgen,amstext,amsbsy,amsopn,amsthm,indentfirst}
\usepackage{lineno}%显示数字导言区
\allowdisplaybreaks%[4]

\newtheorem{theorem}{Theorem}[section]

\newtheorem{defi}[theorem]{Definition}

\newtheorem{lemma}{Lemma}[section]

\newtheorem{problem}[theorem]{Problem}

\begin{document}
%\linenumbers %显示数字
\title
{\LARGE \textbf{ Further results on the permanental sums  of bicyclic  graphs}}

\author{ Tingzeng Wu\thanks{\emph{E-mail address}: mathtzwu@163.com}, Yinggang Bai\\
{\small School of Mathematics and Statistics, Qinghai Nationalities University,}\\
{\small  Xining, Qinghai 810007, P.R.~China}}

\date{}

\maketitle

\noindent {\bf Abstract: } Let $G$ be a graph, and let $A(G)$ be the adjacency matrix of  $G$. The permanental
polynomial of $G$ is defined as $\pi(G,x)=\mathrm{per}(xI-A(G))$. The  permanental sum of   $G$  can be defined as the sum of  absolute value of  coefficients of $\pi(G,x)$. Computing  the permanental sum is $\#$P-complete. Any a bicyclic graph  can be generated from three types of induced subgraphs.
In this paper,  we determine the upper bound of permanental sums of  bicyclic graphs generated from each a type of induced subgraph. And we also determine the second maximal permanental sum of all  bicyclic graphs.


\section{Introduction}

The {\em permanent} of $n\times n$ matrix $X=(x_{ij})(i,j=1,2,\ldots,n)$ is defined as $${\rm per}(X)=\sum_{\sigma}\prod_{i=1}^{n}x_{i\sigma(i)},$$
where the sum is taken over all permutations $\sigma$ of $\{1,2,\ldots, n\}$.

Let $G$ be a graph with vertex set $V(G)$ and edge set $E(G)$. The number of   vertices of $G$ is called its order. Particularly, $G$ is called {\em empty graph} if $V(G)=\emptyset$.
Let  $A(G)$  and $I$  denote the  adjacency matrix and  unit matrix, respectively.  The {\em permanental polynomial} of $G$ is defined as
\begin{eqnarray*}
\pi(G,x)=\mathrm{per}(xI-A(G))=\sum_{k=0}^{n}b_{k}(G)x^{n-k},
\end{eqnarray*}
and its theory is well elaborated \cite{cas1,kas,shi,yan,zha}.

A {\em Sachs graph} is a graph in which each component is a single edge or a cycle. Merris et al.\cite{mer} given  the Sachs type result concerning the coefficients of the permanental polynomial of $G$, i.e.,
\begin{eqnarray*}\label{equ1}
b_{k}(G)=(-1)^{k}\sum_{H}2^{c(H)},~~ 0\leq k \leq n,
\end{eqnarray*}
where the sum is taken over all Sachs subgraphs $H$ of $G$ on $k$ vertices and $c(H)$ is the
number of cycles in $H$.

The {\em permanental sum} of a graph  $G$, denoted by  $PS(G)$, can be defined as the summation of all absolute values of  coefficients of permanental polynomial of $G$, and
 it can be expressed by the
following equation:
\begin{eqnarray*}\label{equ2}
PS(G)=\sum\limits_{i=0}^{n}|b_{i}(G)|
=\sum\limits_{i=0}^{n}\sum_{H}2^{c(H)},
\end{eqnarray*}
where the second sum is taken over all Sachs subgraphs $H$ of $G$ on $k$ vertices and $c(H)$ is the
number of cycles in $H$.
 Specifically, $PS(G)=1$ if $G$ is an empty graph. Wu and So\cite{wuso1} have shown that computing  the permanental sum is $\#$P-complete.

Studies on the permanental sum of a graph  is of great interest in both chemistry and mathematics. The permanental sum of a graph was first considered by Tong\cite{ton}.  Xie et al.\cite{xie} captured a labile fullerene $C_{50}(D_{5h})$. Tong computed all 271 fullerenes in $C_{50}$. In
his study, Tong found that the permanental sum of $C_{50}(D_{5h})$
achieves the minimum among all 271 fullerenes in $C_{50}$. He
pointed that the permanental sum would be closely related to
stability of molecular graphs.
  The permanental sum of a graph is also meaningful in
mathematics since it is related to the enumeration of Sachs subgraphs of a graph. Recently,  Wu and Lai\cite{wu1} systematically introduced the properties of permanental sum of a graph.  Li et al.\cite{liqz} determined the extremal hexagonal chains with respect to permanental sum. Li and Wei\cite{liw} study the property of permanental sum of  octagonal chains. For the background and some known results about permanental sum, we refer the reader to \cite{so, wuso2} and the references therein.

%Wu et al. \cite{wu4} determined the smallest bound of  permanental sum of bipartite graphs.

A {\em bicyclic graph} is a simple connected graph where the number of edges equals the number
of vertices plus one.  Let $\mathscr{B}_{n}$ be the collection of all bicyclic graphs with $n$ vertices. By the structures of bicyclic graphs, it can be known that $\mathscr{B}_{n}$ consists of three types of graphs: first type, denoted by
$\mathscr{B}^{1}_{n}(p,q)$, is the set of those graphs each of which contains  $B_{1}(p,q)$ as  vertex-induced subgraph; second type, denoted by
$\mathscr{B}^{2}_{n}(p,q,r)$, is the set of those graphs each of which contains  $B_{2}(p,q,r)$ as  vertex-induced subgraph; third type, denoted by
$\mathscr{B}^{3}_{n}(p,q,r)$, is the set of those graphs each of which contains  $B_{3}(p,q,r)$ as  vertex-induced subgraph. Obviously, $\mathscr{B}_{n}=\mathscr{B}^{1}_{n}(p,q)\cup\mathscr{B}^{2}_{n}(p,q,r)\cup\mathscr{B}^{3}_{n}(p,q,r)$.

\begin{figure}[htbp]
\begin{center}
\includegraphics[scale=0.55]{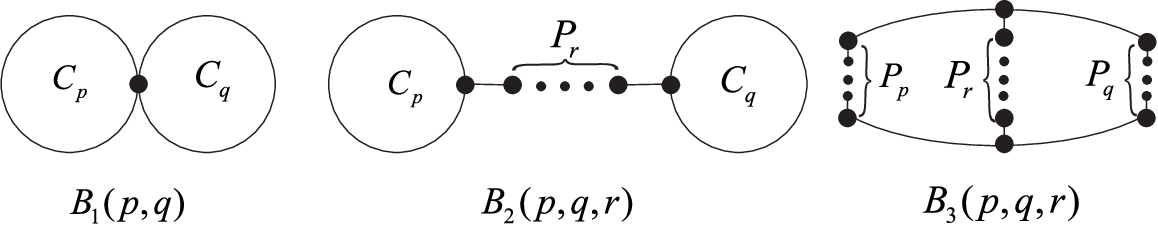}
\caption{\label{fig11}\small {Bicyclic graphs  $B_{1}(p,q)$,   $B_{2}(p,q,r)$ and  $B_{3}(p,q,r)$.}}
 \end{center}
 \end{figure}

In \cite{wud}, the authors investigated the properties of permanental sums of bicyclic graphs. And the
bounds of permanental sums of bicyclic graphs are determined. Furthermore, the authors presented two problems. That is,
\begin{problem}
Characterizing the extremal graphs with maximum permanental sum in $\mathscr{B}^{1}_{n}(p,q)$ and $\mathscr{B}^{3}_{n}(p,q,r)$, respectively.
\end{problem}
\begin{problem}
Determining the second  maximal  permanental sums of bicyclic graphs in $\mathscr{B}_{n}$.
\end{problem}

In this paper, we give the solutions of the problems as above. And we obtain the results as follows.
\begin{theorem}\label{art1}
Let $G\in\mathscr{B}^{1}_{n}(p,q)$ be a bicyclic graph with $n\geq7$ vertices. Then $PS(G)\leq 8F(n-3)+12F(n-4)$ and the equality holds if and only if $G\cong H$, where $H$ obtained by connecting a vertex of degree 2 of  $B_{1}(3,3)$ and a pendant vertex of $P_{n-5}$.
\end{theorem}
\begin{theorem}\label{art2}
Let $G\in\mathscr{B}^{2}_{n}(p,q,r)-B_{2}(3,3,n-6)$ be a bicyclic graph with $n\geq11$ vertices. Then $PS(G)\leq 36F(n-5)+24F(n-6)$ and the equality holds if and only if $G\cong B'_{2}(3,3,1,n-7)$.
\end{theorem}
\begin{theorem}\label{art3}
Let $G\in\mathscr{B}^{3}_{n}(p,q,r)$ be a bicyclic graph with $n\geq6$ vertices. Then $PS(G)\leq 14F(n-3)+6F(n-4)$ and the equality holds if and only if $G\cong H$, where $H$ obtained by connecting a vertex of degree 2 of  $B_{3}(1,1,0)$ and a pendant vertex of $P_{n-4}$.
\end{theorem}
By Theorems \ref{art1},  \ref{art2} and  \ref{art3}, we can obtain that
\begin{theorem}\label{art4}
Let $G\in\mathscr{B}_{n}-B_{2}(3,3,n-6)$ be a bicyclic graph with $n\geq11$ vertices. Then $PS(G)\leq 36F(n-5)+24F(n-6)$ and the equality holds if and only if $G\cong H$, where $H$ obtained by joining the vertex of degree 2 of $P_{1}$ in $B_{2}(3,3,1)$ to a pendant vertex of $P_{n-7}$.
\end{theorem}
The rest of this paper is organized as follows. In Section 2, we present some preliminaries
and a list of some previously known results about  permanental sums of graphs. In Section 3,
we give the proof of Theorem \ref{art1}. In Section 4, we give the proof of Theorem \ref{art2}. In final section, we give the proof of Theorem \ref{art3}.

\section{Preliminaries}

In this section, we present some definitions and lemmas which are  important roles in the later. First, we give the definitions of three types of graphs as follows.
\begin{itemize}
\item Let $D(r, n-r)$ be the graph obtained by attaching an end vertex   of $P_{n-r}$ to one vertex of $C_{r}$, where $n\geq 4$ and $3\leq r\leq n-1$.

\item Let $D'(3,r,n-r-3)$ be a unicyclic graph with $n$ vertices obtained by identifying the pendant vertex of $D(3,1)$ and the $r$-th vertex of $P_{n-3}=v_{1}v_{2}\ldots v_{\lfloor\frac{n-3}{2}\rfloor}$. In particular, $D'(3,1,n-4)=D(3, n-3)$.

\item Let $B'_{2}(3,3,1,n-7)$ be the graph obtained by the vertex of $P_{1}$ in $B_{2}(3,3,1)$ attaching a pendant vertex of $P_{n-7}$.

\end{itemize}

\begin{lemma}\label{art21} {\rm \cite{wu1}}
 The permanental sum of a graph satisfies the following identities: \\
 $(i)$ Let $G$ and $H$ be two connected graphs. Then
 \begin{eqnarray*}
 PS(G\cup H) = PS(G)PS(H).
 \end{eqnarray*}
 $(ii)$  Let $uv$ be an edge of  graph $G$ and ${\cal C}(uv)$ the set of cycles containing $uv$. Then
 \begin{eqnarray*}
 PS(G) = PS(G-uv)+PS(G-v-u)+2\sum_{C_k \in {\cal C}(uv)}PS(G-V(C_k)).
 \end{eqnarray*}
 $(iii)$ Let $v$ be a vertex of  graph $G$ and ${\cal C}(v)$ the set of cycles containing $v$. Then
 \begin{eqnarray*}
 PS(G) = PS(G-v)+\sum_{u \in N_{G}(v)}PS(G-v-u)+2\sum_{C_{k} \in {\cal C}(v)}PS(G-V(C_k)).
 \end{eqnarray*}
 \end{lemma}

\begin{defi}(\cite{wuso1})\label{art22}
Let $v$ be  a vertex of  a graph $G\neq P_{1}$.
$G_{1}$ denotes the graph obtained from identifying $v$ with the vertex $u$ of a tree $T$ of order $|T|>2$.
$G_{2}$  is obtained from  identifying $v$ with the pendant vertex $u$ of a path $P_{|T|}$.
We designate the transformation from $G_{1}$
to $G_{2}$ as type {\bf I}.
\end{defi}
\begin{theorem}(\cite{wuso1})\label{art23}
Let $G_{1}$ and $G_{2}$ be the graphs  defined in
Definition \ref{art22}. Then $PS(G_{1})<PS(G_{2})$.
\end{theorem}

\begin{defi}\label{art24}
Let $G$ be a graph, and let   $P_{p+2}=uu_{1}u_{2}\ldots u_{p}v\in V(G)$ be a path in $G$ with  $d(u_{i})=2$, where $i=1,2,\ldots, p$ and $p\geq 2$. Assume that  $G_{1}$ is a graph obtained by respectively attaching two pendant vertices of path $P_{s}(s\geq1)$ and $P_{t}(t\geq1)$ to $u_{i}$ and $u_{p}$ of $P_{p}$ in $G$.  Suppose that  $G_{2}$ is a graph obtained from $G_{1}$ by  deleting the path $P_{s}$ and attaching the path $P_{s}$ to the end of $P_{t}$ in $G_{1}$. Suppose that  $G_{3}$ is a graph obtained from $G_{1}$ by  deleting the path $P_{t}$ and attaching the path $P_{t}$ to the end of $P_{s}$ in $G_{1}$. We designate the transformation from $G_{1}$
to $G_{2}$ or $G_{3}$ as type {\bf II}.
\end{defi}
\begin{theorem}\label{art25}
Let $G_{i}$ be the graphs  defined in
Definition \ref{art24}. Then $PS(G_{1})<PS(G_{2})$ or $PS(G_{1})<PS(G_{3})$.
\end{theorem}
\begin{proof}
For convenience, suppose that $W_{1}=PS(G-\{u_{1}, u_{2},\ldots, u_{p}\})$, $W_{2}=PS(G-\{u_{1}, u_{2},\ldots,\\ u_{p}, u\})$, $W_{3}=PS(G-\{u_{1}, u_{2},\ldots, u_{p}, v\})$ and $W_{4}=PS(G-\{u_{1}, u_{2},\ldots, u_{p}, u, v\})$.

By Lemma \ref{art21}, we obtain that
\begin{eqnarray}\label{equ21}
PS(G_{1})&=&PS(G_{1}-uu_{1})+PS(G_{1}-\{u,u_{1}\})+2\sum_{C\in~{\cal C}_{G_{1}(uu_{1})}} PS(G_{1}-V(C))\nonumber\\
&=&PS(G_{1}-uu_{1}-u_{p}v)+PS(G_{1}-uu_{1}-\{u_{p},v\})+PS(G_{1}-\{u,u_{1}\}-u_{p}v)\nonumber\\
&&+PS(G_{1}-\{u,u_{1}\}-\{u_{p},v\})+2\sum_{C\in~{\cal C}_{G_{1}(uu_{1})}} PS(G_{1}-V(C))\nonumber\\
&=&W_{1}PS(P_{s+p+t})+W_{3}PS(P_{s+p-1})PS(P_{t})+W_{2}PS(P_{s})PS(P_{p+t-1})\\
&&+W_{4}PS(P_{s})PS(P_{t})PS(P_{p-2})+2PS(P_{t})PS(P_{s})\sum_{C\in~{\cal C}_{G(uu_{1})}} PS(G-V(C)),\nonumber
\end{eqnarray}

\begin{eqnarray}\label{equ22}
PS(G_{2})&=&PS(G_{2}-uu_{1})+PS(G_{2}-\{u,u_{1}\})+2\sum_{C\in~{\cal C}_{G_{2}(uu_{1})}} PS(G_{2}-V(C))\nonumber\\
&=&PS(G_{2}-uu_{1}-u_{p}v)+PS(G_{2}-uu_{1}-\{u_{p},v\})+PS(G_{2}-\{u,u_{1}\}-u_{p}v)\nonumber\\
&&+PS(G_{2}-\{u,u_{1}\}-\{u_{p},v\})+2\sum_{C\in~C_{G_{2}(uu_{1})}} PS(G_{2}-V(C))\nonumber\\
&=&W_{1}PS(P_{s+p+t})+W_{3}PS(P_{p-1})PS(P_{s+t})+W_{2}PS(P_{p+t+s-1})\\
&&+W_{4}PS(P_{s+t})PS(P_{p-2})+2PS(P_{t+s})\sum_{C\in~{\cal C}_{G(uu_{1})}} PS(G-V(C))\nonumber
\end{eqnarray}
and
\begin{eqnarray}\label{equ23}
PS(G_{3})&=&PS(G_{3}-uu_{1})+PS(G_{3}-\{u,u_{1}\})+2\sum_{C\in~{\cal C}_{G_{3}(uu_{1})}} PS(G_{3}-V(C))\nonumber\\
&=&PS(G_{3}-uu_{1}-u_{p}v)+PS(G_{3}-uu_{1}-\{u_{p},v\})+PS(G_{3}-\{u,u_{1}\}-u_{p}v)\nonumber\\
&&+PS(G_{3}-\{u,u_{1}\}-\{u_{p},v\})+2\sum_{C\in~{\cal C}_{G_{3}(uu_{1})}} PS(G_{3}-V(C))\nonumber\\
&=&W_{1}PS(P_{s+p+t})+W_{3}PS(P_{s+t+p-1})+W_{2}PS(P_{s+t})PS(P_{p-1})\\
&&+W_{4}PS(P_{s+t})PS(P_{p-2})+2PS(P_{t+s})\sum_{C\in~{\cal C}_{G(uu_{1})}} PS(G-V(C)).\nonumber
\end{eqnarray}

Suppose that $W_{2}> W_{3}$. By Lemma \ref{art21},  (\ref{equ21}) and (\ref{equ22}), we have
\begin{eqnarray}\label{equ24}
&&PS(G_{2})-PS(G_{1})\nonumber\\
&=&W_{3}[PS(P_{p-1})PS(P_{t+s})-PS(P_{s+p-1})PS(P_{t})]
+W_{2}[PS(P_{t+p+s-1})-PS(P_{s})\nonumber\\
&&\times PS(P_{p+t-1})]
+W_{4}[PS(P_{p-2})PS(P_{s+t})-PS(P_{p-2})PS(P_{s})PS(P_{t})]
+2[PS(P_{t+s}))\nonumber\\
&&-PS(P_{t})PS(P_{s})]\sum_{C\in~{\cal C}_{G(uu_{1})}} PS(G-V(C))\nonumber\\
&=&W_{3}[PS(P_{s-1})PS(P_{t-1})PS(P_{p-1})-PS(P_{s-1})PS(P_{t})PS(P_{p-2})]+W_{2}[PS(P_{s-1})\nonumber\\
&&\times PS(P_{t-1})PS(P_{p-1})+PS(P_{s-1})PS(P_{t-2})PS(P_{p-2})]
+W_{4})PS(P_{p-2})PS(P_{s-1})\nonumber\\
&&\times PS(P_{t-1})+2PS(P_{t-1})PS(P_{s-1})\sum_{C\in~{\cal C}_{G(uu_{1})}}PS(G-V(C))\nonumber\\
&=&W_{3}[F(s)F(t)F(p)-F(s)F(t+1)F(p-1)]
+W_{2}[F(s)F(t)F(p)+F(s)F(t-1)\nonumber\\
&&\times F(p-1)]+W_{4}F(p-1)F(s)F(t)
+2F(t)F(s)\sum_{C\in~{\cal C}_{G(uu_{1})}} PS(G-V(C)).
\end{eqnarray}
By (\ref{equ24}), if $F(s)F(t)F(p)-F(s)F(t+1)F(p-1)\geq0$ then $PS(G_{2})>PS(G_{1})$, if $F(s)F(t)F(p)-F(s)F(t+1)F(p-1)<0$ then \begin{eqnarray*}
PS(G_{2})-PS(G_{1})&>&W_{2}[F(s)F(t)F(p)-F(s)F(t+1)F(p-1)+F(s)F(t)F(p)\\
&&+F(s)F(t-1)F(p-1)]+W_{4}F(p-1)F(s)F(t)\\
&&+2F(t)F(s)\sum\limits_{C\in~{\cal C}_{G(uu_{1})}} PS(G-V(C))\\
&=&W_{2}[2F(s)F(t)F(p)-F(s)F(p-1)F(t)]+W_{4}F(p-1)F(s)F(t)\\
&&+2F(t)F(s)\sum\limits_{C\in~{\cal C}_{G(uu_{1})}} PS(G-V(C))>0.
\end{eqnarray*}
Thus, $PS(G_{2})>PS(G_{1})$.

Suppose that $W_{2}\leq W_{3}$. By Lemma \ref{art21},  (\ref{equ21}) and (\ref{equ23}), we have

\begin{eqnarray}\label{equ25}
&&PS(G_{3})-PS(G_{1})\nonumber\\
&=&W_{3}[PS(P_{s+t+p-1})-PS(P_{s+p-1})PS(P_{t})]
+W_{2}[PS(P_{s+t})PS(P_{p-1})\nonumber\\
&&-PS(P_{s})PS(P_{p+t-1})]
+W_{4}[PS(P_{s+t})PS(P_{p-2})-PS(P_{s})PS(P_{t})PS(P_{p-2})]\nonumber\\
&&+2[PS(P_{s+t})-PS(P_{s})PS(P_{t})\sum_{C\in~{\cal C}_{G(uu_{1})}} PS(G-V(C))\nonumber\\
&=&W_{3}[PS(P_{s-1})PS(P_{t-1})PS(P_{p-1})+PS(P_{s-2})PS(P_{t-1})PS(P_{p-2})]
+W_{2}\nonumber\\
&&\times[PS(P_{s-1})PS(P_{t-1})PS(P_{p-1})
-PS(P_{s})PS(P_{t-1})PS(P_{p-2})]
+W_{4}PS(P_{p-2})\nonumber\\
&&\times PS(P_{s-1})PS(P_{t-1})+2PS(P_{s-1})PS(P_{t-1})\sum_{C\in~{\cal C}_{G(uu_{1})}} PS(G-V(C))\nonumber\\
&=&W_{3}[F(s)F(t)F(p)+F(s-1)F(t)F(p-1)]+
W_{2}[F(s)F(t)F(p)-F(s+1)F(t)\nonumber\\
&&F(p-1)]+
W_{4}F(p-1)F(s)F(t)+
2F(s)F(t)\sum_{C\in~C_{G(uu_{1})}} PS(G-V(C)).
\end{eqnarray}
By (\ref{equ25}), if $F(s)F(t)F(p)-F(s+1)F(t)F(p-1)\geq0$ then $PS(G_{3})>PS(G_{1})$, if $F(s)F(t)F(p)-F(s+1)F(t)F(p-1)<0$ then
\begin{eqnarray*}
PS(G_{3})-PS(G_{1})&\geq& W_{3}[F(s)F(t)F(p)-F(s+1)F(t)F(p-1)+F(s)F(t)F(p)\\
&&+F(s-1)F(t)F(p-1)]
+W_{4}F(p-1)F(s)F(t)\\
&&+2F(s)F(t)\sum\limits_{C\in~{\cal C}_{G(uu_{1})}} PS(G-V(C))\\
&=&W_{3}[2F(s)F(t)F(p)-F(s)F(t)F(p-1)]
+W_{4}[F(p-1)F(s)F(t)\\
&&+2F(s)F(t)\sum\limits_{C\in~C_{G(uu_{1})}} PS(G-V(C))>0.
\end{eqnarray*}
So, $PS(G_{3})>PS(G_{1})$.
\end{proof}

\begin{lemma}\label{art210} {\rm \cite{wu1}}
$F(n)=F(k)F(n-k+1)+F(k-1)F(n-k)$ for $1\leq k\leq n$.
\end{lemma}

\begin{lemma}\label{newwul2} {\rm \cite{wu1}}
Let $G$ be a forest with $n$ vertices. Then $PS(G)\leq PS(P_{n})$ with equality holding if and only if $G\cong P_{n}$.
\end{lemma}

\begin{lemma}\label{new3}
Let $H$ be a forest of order $n(\geq 4)$ containing at least two components. Then
\begin{eqnarray*}
PS(H)\leq 2F(n-1),
\end{eqnarray*}
where the equality holds if and only if $H=P_{2}\cup P_{n-2}$.
\end{lemma}
\begin{proof}
Suppose that $H$ has $2$ or $3$ vertices. The proof is straightforward. The following we set $n\geq 4$, and we consider two cases.

\textbf{Case 1.} Suppose that $H=T_{m}\cup~T_{t}$ contains two components. Without loss of generality,
Assume that $ m\leq t$ and $m+t=n$.
By Lemmas \ref{art21} and \ref{newwul2}, we obtain that
 $PS(H)=PS(T_{m})PS(T_{t})\leq F(m+1)F(t+1)$
with equality holding if and only if $H=P_{m}\cup P_{t}$, and $PS(P_{2}\cup P_{n-2})=2F(n-1)$.
Set $m\neq 2$. By Lemma \ref{art210}, we get that $F(m+1)F(t+1)-2F(n-1)=[F(m-1)-F(m)][F(t-1)-F(t)]<0$. So, $PS(T_{m}\cup T_{t})\leq2F(n-1)$, where the equality holds if and only if $T_{m}\cup T_{t}=_{2}\cup~P_{n-2}$.

\textbf{Case 2.} Assume that $H=T_{1}\cup~T_{2}\cup\cdots\cup T_{i}$ contains at least three components, and assume that  $T_{i}$ has $|T_{i}|$ vertices. By Lemma \ref{newwul2}, we have
$PS(H)<PS(P_{|T_{1}|}\cup~P_{|T_{2}|+\cdots+|T_{i}|})$.
By Lemma \ref{art21}, we get that $PS(H)<2F(n-1) $.
\end{proof}

\begin{lemma}\label{newwl} {\rm \cite{wu1}}
Let $G\in\mathscr U_{n}$ be a graph with $n\geq5$ vertices. Then $PS(G)\leq~6F(n-2)+2F(n-3)$ with equality holding if and only if $G\cong~D(3,n-3)$.
\end{lemma}

\begin{lemma}\label{newwll} {\rm \cite{wuso1}}
Let $G\in\mathscr U_{n}-D(3,n-3)$ be a unicyclic graph with $n$ vertices.
 If $n\geq8$, then $PS(G)\leq~6F(n-2)+4F(n-5)$, where  the equality holds if and only if $G\cong~D^{'}(3,3,n-6)$.
\end{lemma}

\section{ The proof of Theorem \ref{art1}}
Let $G\in\mathscr{B}^{1}_{n}(p,q)$ be a bicyclic graph with $n$ vertices. Repeatedly applying transformations {\bf I} and {\bf II}, we get that the structure of $G$ is isomorphic to one of $B_{1}(p,q)$, $B_{1}^{1}(p,q,r)$, $B_{1}^{2}(p,q,r)$, $B_{1}^{3}(p,q,r,s)$, $B_{1}^{4}(p,q,r,s)$ and $B_{1}^{5}(p,q,r,s,t)$, see Figure \ref{fig31}. The following we prove that  the permanental sums of these graphs are less than or equal to the permanental sum of $B_{1}^{2}(3,3,n-5)$.
\begin{figure}[htbp]
\begin{center}
\includegraphics[scale=0.45]{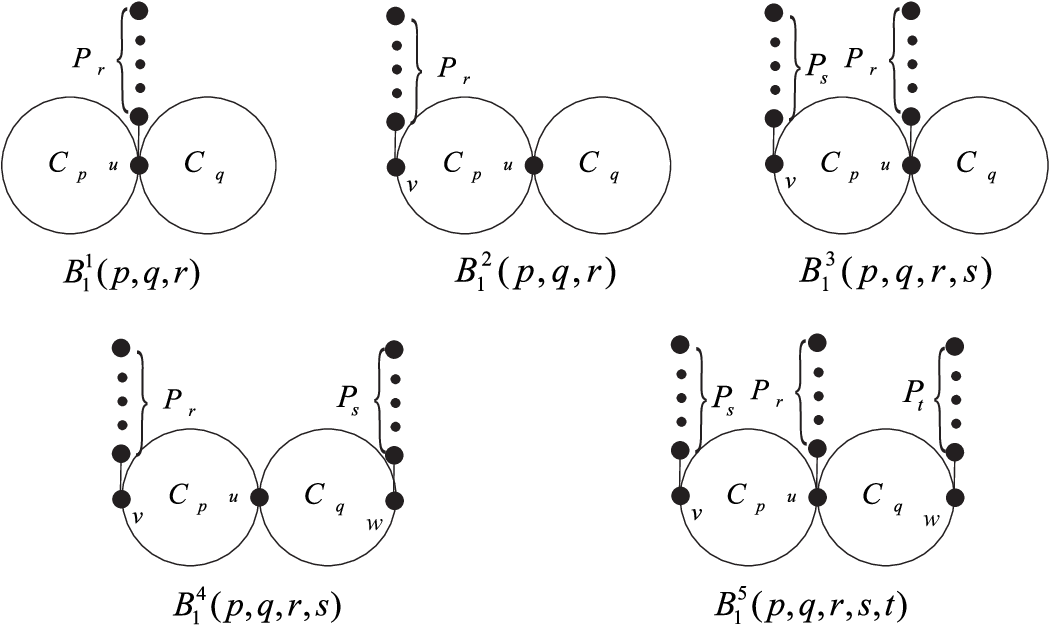}
\caption{\label{fig31}\small {Bicyclic graphs   $B_{1}^{1}(p,q,r)$, $B_{1}^{2}(p,q,r)$, $B_{1}^{3}(p,q,r,s)$, $B_{1}^{4}(p,q,r,s)$ and $B_{1}^{5}(p,q,r,s,t)$.}}
 \end{center}
 \end{figure}
\begin{lemma}\label{art35}
Let $G\in\mathscr B_{n}^{1}(p,q)$ be a graph with $n\geq6$ vertices, and let $w$ be a vertex in induced subgraph $B_{1}(p,q)$ of $G$. Then $PS(G-w)<PS(B_{1}^{2}(3,3,n-6))$.
\end{lemma}
\begin{proof}
By Lemma \ref{art21}, we have
\begin{eqnarray}\label{equ31}
PS(B_{1}^{2}(3,3,n-6))>PS(D(3,n-4)).
\end{eqnarray}

According to the structure of $G-w$, we consider three cases as follows.

Assume that $G-w$ is a unicyclic graph with $n-1$~vertices. By Lemma \ref{newwl} and (\ref{equ31}), we have $PS(B_{1}^{2}(3,3,n-6))>PS(G-w)$.

Suppose that $G-w$ is the disjoint union of forest $H$ with  $n-m-1$ vertices and unicyclic graph $U_{m}$, where $3\leq m\leq n-2$. By Lemmas \ref{art21} and \ref{newwul2}, we have
$PS(G-w)=PS(U_{m}\cup H)\leq PS(U_{m})PS(P_{n-m-1})<PS(M)$, where $M$ be a graph obtained by joining  a vertex of $U_{m}$ and a pendant vertex of $P_{n-m-1}$.
By Lemma \ref{newwl} and (\ref{equ31}), we obtain that  $PS(G-w)<PS(B_{1}^{2}(3,3,n-6))$.

Assume that $G-w$ is  a forest $H$ with $n-1$ vertices. By Lemmas \ref{art21}, \ref{newwul2} and (\ref{equ31}), $PS(G-w)=PS(H)\leq~PS(P_{n-1})<PS(D(3,n-4))<PS(B_{1}^{2}(3,3,n-6))$. \end{proof}

\begin{lemma}\label{art31}
Let $G\cong B_{1}(p,q)$ be a graph with $n$ vertices.\\
$(i)$ If $n=6$, then $PS(G)=PS(B_{1}(3,4))>B_{1}^{2}(3,3,1)$.\\
 $(ii)$ If $n\geq 7$, then $PS(G)<PS(B_{1}^{2}(3,3,n-5))$.
\end{lemma}
\begin{proof}
Direct calculation yields $30=PS(B_{1}(3,4))>PS(B_{1}^{2}(3,3,1)=28$.
By Lemmas \ref{art21} and \ref{art210}, we have
$PS(B_{1}^{2}(3,3,n-5))-PS(B_{1}(p,q))
=F(n-1-q)[5F(q-1)+11F(q-2)-2]+F(n-2-q)[4F(q-2)+10F(q-3)]-2F(q)>F(q-1)+11F(q-2)+10F(q-3)-2>0$.
\end{proof}

\begin{lemma}\label{art32}
Let $G\cong B_{1}^{2}(p,q,r)$ be a graph with $n$ vertices. Then $PS(G)\leq PS(B_{1}^{2}(3,3,n-5))$, where the equality holds if and only if $G\cong B_{1}^{2}(3,3,n-5)$.
\end{lemma}
\begin{proof}
Let $G\neq B_{1}^{2}(3,3,n-5)$. By Lemma \ref{art21},
we get that
\begin{eqnarray}\label{equ30}
PS(B_{1}^{2}(3,3,n-5))-PS(G)
&=&PS(P_{r})[PS(B_{1}^{2}(3,3,n-r-5))-PS(B_{1}(p,q))]\\
&&+PS(P_{r-1})[PS(B_{1}^{2}(3,3,n-r-6))-PS(B_{1}(p,q)-v)].\nonumber
\end{eqnarray}

Suppose that $p+q=7$ in $G$. It is easy to know that $B_{1}(4,3)-v$ is a unicyclic graph with $5$ vertices. By Lemma \ref{newwl},~we get that $PS(B_{1}(4,3)-v))\leq~PS(D(3,2))=14$, where the quality holds if and only if $B_{1}(4,3)-v\cong~D(3,2)$. By Lemma \ref{art31}$(i)$ and (\ref{equ30}), we have
\begin{eqnarray*}
PS(B_{1}^{2}(3,3,n-5))-PS(G)
&=&PS(P_{n-6})[PS(B_{1}^{2}(3,3,1))-PS(B_{1}(3,4))]\\
&&+PS(P_{n-7})[PS(B_{1}(3,3))-PS(B_{1}(3,4)-v)]\\
&\geq&-2F(n-5)+6F(n-6)>0.
\end{eqnarray*}

Assume that $p+q\geq8$ in $G$. By Lemmas \ref{art21}, \ref{art35}, \ref{art31}$(ii)$ and (\ref{equ30}), we obtain that
\begin{eqnarray*}
PS(B_{1}^{2}(3,3,n-5))-PS(G)
&=&PS(P_{r})[PS(B_{1}^{2}(3,3,n-r-5))-PS(B_{1}(p,q))]\\
&&+PS(P_{r-1})[PS(B_{1}^{2}(3,3,n-r-6))-PS(B_{1}(p,q)-v)]>0.
\end{eqnarray*}
\end{proof}

\begin{lemma}\label{art33}
Let $B_{1}^{1}(p,q,r)$  and  $B_{1}^{2}(p',q',r')$ be two graphs with $n$ vertices, see Figure \ref{fig31}. If $p=p'$, $q=q'$ and $r=r'$,  then $PS(B_{1}^{1}(p,q,r))<PS(B_{1}^{2}(p',q',r'))$.
\end{lemma}
\begin{proof}
By Lemma \ref{art21}, we get that
$PS(B_{1}^{2}(p,q,r))-PS(B_{1}^{1}(p,q,r))
=[PS(B_{1}(p,q)-v))-PS(B_{1}(p,q)-u))]PS(P_{r-1})$.
Checking $B_{1}(p,q)-v$ and $B_{1}(p,q)-u$, we know that $B_{1}(p,q)-u$ is a subgraph of $B_{1}(p,q)-v$. By Lemma \ref{art21}$(ii)$, $PS(B_{1}(p,q,r)-v)>PS(B_{1}(p,q,r)-u)$. So, $PS(B_{1}^{2}(p,q,r))-PS(B_{1}^{1}(p,q,r))>0$.
\end{proof}
\begin{lemma}\label{art34}
Let $B_{1}^{2}(p,q,r)$  and  $B_{1}^{3}(p',q',r',s')$ be two graphs with $n$ vertices, see Figure \ref{fig31}. If $p=p'$, $q=q'$ and $r=r'+s'$,  then $PS(B_{1}^{3}(p',q',r',s'))<PS(B_{1}^{2}(p,q,r))$.
\end{lemma}
\begin{proof}
Suppose that $P_{r}=w_{1}w_{2}\ldots w_{s'}w_{s'+1}\ldots w_{r'+s'}$ is an induced path of $B_{1}^{2}(p,q,r)$. By Lemma \ref{art21}, we have
$PS(B_{1}^{2}(p,q,r))-PS(B_{1}^{3}(p',q',r',s'))=[PS(B_{1}^{2}(p,q,s'-1))-PS(B_{1}^{2}(p,q,s')-u)]PS(P_{r-1})$.
By Lemma \ref{newwul2} and the structure of $B_{1}^{2}(p,q,s')-u$, we have $PS(B_{1}^{2}(p,q,s')-u)\leq PS(P_{q-1}\cup P_{s'+p-1})$. By Lemma \ref{art21},
$PS(B_{1}^{2}(p,q,s'-1))>PS(D(p,s'-1)\cup P_{q-1})>PS(P_{p+s'-1}\cup P_{q-1})$. Thus, $PS(B_{1}^{2}(p,q,r))-PS(B_{1}^{3}(p',q',r',s'))>0$.
\end{proof}

\begin{lemma}\label{art36}
Let $B_{1}^{2}(p,q,r)$  and  $B_{1}^{4}(p',q',r',s')$ be two graphs with $n$ vertices, see Figure \ref{fig31}. If $p=p'$, $q=q'$ and $r=r'+s'$,  then $PS(B_{1}^{4}(p',q',r',s'))<PS(B_{1}^{2}(p,q,r))$.
\end{lemma}
\begin{proof}
 By Lemmas \ref{art21}, \ref{art35} and \ref{art32}, we obtain that $PS(B_{1}^{4}(p',q',r',s'))
=PS(B_{1}^{2}(p,q,r'))PS(P_{s'})+PS(B_{1}^{2}(p,q,r')-w)PS(P_{s'-1})
<PS(B_{1}^{2}(3,3,n-s'-5))PS(P_{s'})+PS(B_{1}^{2}(3,3,n-s'-6))PS(P_{s'-1})=PS(B_{1}^{2}(3,3,n-5))$.
\end{proof}
\begin{lemma}\label{art37}
Let $B_{1}^{2}(p,q,r)$  and  $B_{1}^{5}(p',q',r',s',t')$ be two graphs with $n$ vertices, see Figure \ref{fig31}. If $p=p'$, $q=q'$ and $r=r'+s'+t'$,  then $PS(B_{1}^{5}(p',q',r',s',t'))<PS(B_{1}^{2}(p,q,r))$.
\end{lemma}
\begin{proof}
Similar to the proof of Lemma \ref{art36}, by Lemmas \ref{art21}, \ref{art35} and \ref{art34}, we obtain that $PS(B_{1}^{5}(p',\\q',r',s',t'))
=PS(B_{1}^{3}(p,q,r',s'))PS(P_{t'})+PS(B_{1}^{3}(p,q,r',s')-w)PS(P_{t'-1})
<PS(B_{1}^{2}(3,3,n-t'-5))PS(P_{t'})+PS(B_{1}^{2}(3,3,n-t'-6))PS(P_{t'-1})
=PS(B_{1}^{2}(3,3,n-5))$.
\end{proof}
Combining Definitions \ref{art22} and \ref{art24}, Theorems \ref{art23} and \ref{art25}, and Lemmas \ref{art31},~\ref{art32},~\ref{art33},~\ref{art34},~\ref{art36}~and\\~\ref{art37},~the proof of Theorem \ref{art1} is straightforward.

\section{ The proof of Theorem \ref{art2}}

\begin{lemma}\label{art41}( {\cite{wu1}}) Let~$G\in\mathscr B_{n}^{2}$ be a graph with $n\geq8$ vertices. Then~$PS(G)\leq20F(n-3)$~with equality holding if and only if~$G\cong B_{2}(3,3,n-6)$.
\end{lemma}

\begin{lemma}\label{art42}
Let $U_{m}$ and $U_{t}$ be two unicyclic graphs with $m$ and $t$ vertices, respectively. If $V(U_{m})\cap V(U_{t})=\emptyset$, $m$ and $t\geq 5$, and $m+t=n$, then  $PS(U_{m}\cup U_{t})<54F(n-6)+18F(n-7)$.
\end{lemma}

\begin{proof}
By Lemmas \ref{art21}~and \ref{newwl}, We get that $
PS(U_{m}\cup U_{t})\leq[6F(t-2)+2F(t-3)][6F(m-2)+2F(m-3)]=F(n-t-2)[36F(t-2)+12F(t-3)]
+F(n-t-3)[12F(t-2)+4F(t-3)]$,
where the equality holds if and only if $U_{m}\cup U_{t}=D(3,t-3)\cup D(3,m-3)$.
$54F(n-6)+18F(n-7)-PS(U_{m}\cup U_{t})\geq54F(n-6)+18F(n-7)-[6F(t-2)+2F(t-3)][6F(m-2)+2F(m-3)]
=F(t-4)[26F(n-t-3)-12F(n-t-2)]+F(t-5)[2F(n-t-3)+6F(n-t-2)]>0$.
\end{proof}

\begin{lemma}(\cite{wud})\label{art446}
Let $U_{4}$ be a unicyclic graph with $4$ vertices. Then $PS(U_{4})\leq9$ with equality holding if and only if $U_{4}\cong~C_{4}$.
\end{lemma}

\begin{lemma}\label{art43}
Let $H$ be a forest with $m$ vertices, and let $U_{t}$ be a unicyclic graph with $t$ vertices. If $m\geq1$, $t\geq3$ and $n=m+t\geq7$, then $PS(H\cup U_{t})\leq6F(n-2)$, where the equality holds if and only if $H\cup U_{t}=P_{n-3}\cup C_{3}$.
\end{lemma}

\begin{proof} Suppose that  $t=3$. By Lemma \ref{newwul2}, we have $PS(H\cup U_{t})\leq6F(n-2)$ with equality holding if and only if $H\cup U_{t}=P_{n-3}\cup C_{3}$.

Suppose that $t=4$. By Lemmas \ref{newwul2} and \ref{art446}, we obtain that $PS(H\cup U_{t})\leq9F(n-3)<6F(n-2)$ with equality holding if and only if $H \cup U_{t}=P_{n-4}\cup C_{4}$.

 Assume that $t\geq5$. By Lemmas \ref{art21}, \ref{art210} and \ref{newwl}, we obtain that $PS(H\cup~U_{t})\leq F(m+1)[6F(t-2)+2F(t-3)]<6F(n-2)$ with equality holding if and only if $H\cup U_{t}=P_{m}\cup D(3,t-3)$.
\end{proof}
By simply computing, we have
\begin{lemma}\label{art44}
If $n\geq9$, then $54F(n-6)+18F(n-7)+6F(n-4)\leq36F(n-5)+24F(n-6)
$, where the equality holds if and only if $n=9$.
\end{lemma}

{\bf The proof of Theorem \ref{art2}}.
Let $U_{s}$ and $U_{n-s}$ be two unicyclic graphs with  $s$ vertices and $n-s$ vertices, respectively. Suppose that $u$ is a vertex of induced cycle in $U_{s}$ and $v$ is a vertex of $U_{n-s}$. It is easy to see that  if $G\in\mathscr B_{n}(p,q,r)$ has $n(\geq 11)$ vertices, then $G$ can be obtained by connecting $u$ in $U_{s}$ and $v$ in $U_{n-s}$.
By Lemma \ref{art21}, we get that
\begin{eqnarray}\label{X4}
 PS(G)=PS(G-uv)+PS(G-\{u,v\}).
\end{eqnarray}
The following we discuss the value of $PS(G)$ by (\ref{X4}).

Suppose that $s=3$ in $G$.  Set that $G-\{u,v\}$ is the disjoint union of a forest $H$ with $n-t-2$ vertices and unicyclic graph $U_{t}$ with $t$ vertices, where $3\leq t\leq n-4$. If  the number of components of $H$ is more than one.  By Lemma \ref{newwl}, we get that
$PS(G-uv)=PS(C_{3})PS(U_{n-3})\leq36F(n-5)+12F(n-6)$
 with equality holding if and only if $G-uv=C_{3}\cup~D(3,n-6)$. By Lemma \ref{art43}, we have $PS(G-\{u,v\})\leq~PS(P_{2})\times6F(n-6)=12F(n-6)$
with equality holding if and only if $G-\{u,v\}=P_{2}\cup~P_{n-7}\cup~C_{3}$. According the arguments as above, we get that
$PS(G)\leq36F(n-5)+24F(n-6)$ with equality holding if and only if $G\cong B_{2}'(3,3,1,n-7)$.
If the number of components of $H$ is equal to $1$. Then $G-u-v=P_{2}\cup U_{n-4}$. This implies that
$v$ is the vertex of degree 2 in $G$.
 By Lemma \ref{newwl}, we get
\begin{eqnarray}\label{X9}
PS(G-u-v)=PS(P_{2})PS(U_{n-4})\leq12F(n-6)+4F(n-7)
\end{eqnarray}
with equality holding if and only if $U_{n-4}\cong~D(3,n-7)$. Since $G\ncong B_{2}(3,3,n-6)$, we have
$U_{n-3}\ncong D(3,n-6)$. Thus,
 by Lemma \ref{newwll}, we obtain that
\begin{eqnarray}\label{X10}
PS(G-uv)=PS(C_{3})PS(U_{n-3})\leq36F(n-5)+24F(n-8)
\end{eqnarray}
 with equality holding if and only if $U_{n-3}\cong D'(3,3,n-6)$.  Checking the structure of $G$,
we know that the equality in (\ref{X9}) and (\ref{X10}) cannot be obtained at the same time. Thus,
$PS(G)<36F(n-5)+24F(n-8)+12F(n-6)+4F(n-7)<36F(n-5)+24F(n-6)$.
Set that $G-\{u,v\}$ is a forest with $n-2$ vertices. This implies that  $G-\{u,v\}$ contains at least two components. By Lemma \ref{newwl}, we get that
$PS(G-uv)=PS(C_{3})PS(U_{n-3})\leq36F(n-5)+12F(n-6)$. By Lemma \ref{new3}, we obtain that
$PS(G-\{u,v\})\leq2F(n-3)<12F(n-6)$. Thus $PS(G)<36F(n-5)+24F(n-6)$.

Assume that $s=4$ in $G$. Checking the structure of $G$, we know that $G$ has three possible structures.
One is $H_{1}$ obtained by joining a vertex of $C_{4}$ to a vertex of $U_{4}$.  Others are obtained by connecting a vertex of degree 2 (or 3) in $D(3,1)$ and a vertex of $U_{4}$, denoted by $H_{2}$ and $H_{3}$. Similar to the proof as above,
by   Lemmas \ref{newwul2} and  \ref{art43}, it is easy to obtain that $PS(H_{i}-\{u,v\})<6F(n-4)$, where $i=1, 2, 3$. By Lemma \ref{newwl}, we have $PS(H_{i}-uv)=PS(U_{4})PS(U_{n-4})\leq54F(n-6)+18F(n-7)$, where $i=1, 2, 3$. By Lemma \ref{art44} and arguments as above, we get that $PS(G)<36F(n-5)+24F(n-6)$.

Suppose that $s\geq5$ and $n-s\geq5$ in $G$. By Lemma \ref{art42}, we get that
$PS(G-uv)=PS(U_{n-s})PS(U_{s})<54F(n-6)+18F(n-7)$.
Next, we determine the value of $PS(G-\{u,v\})$.

Assume that $G-\{u,v\}$ is the disjoint union of forest with $s-1$ vertices and unicyclic graph $U_{n-s-1}$ with $n-s-1$ vertices. Since $n-s\geq5$, we get that $U_{n-s-1}\ncong~C_{3}$ and $G-\{u,v\}\ncong~C_{3}\cup~P_{n-5}$.  By Lemma \ref{art43} and $n\geq11$, we obtain that $PS(G-\{u,v\})<6F(n-4)$. By Lemma \ref{art44}, $PS(G)<36F(n-5)+24F(n-6)$.

Assume that $G-\{u,v\}$ is a forest with $n-2$ vertices. Then $G-\{u,v\}$ contains at least two components.  By Lemma \ref{new3}, we have $PS(G-\{u,v\})\leq2F(n-3)<6F(n-4)$. Thus, by Lemma \ref{art44}, we have $PS(G)<36F(n-5)+24F(n-6)$.

Assume that $G-\{u,v\}$ is the disjoint union of a forest with $n-t-2$ vertices and unicyclic graph $U_{t}$ with $t$ vertices, where $3\leq t\leq n-s-2$. By Lemma \ref{art43}, we have $PS(G-\{u,v\})<6F(n-4)$. By Lemma \ref{art44}, $PS(G)<36F(n-5)+24F(n-6)$.

The proof is completed. $\Box$

\section{ The proof of Theorem \ref{art3}}
\begin{lemma}\label{art51}(\cite{wu1})
Let $G$ be a unicyclic graph with $n$ vertices and girth $r$. Then\\
$(i)$ $PS(G)\leq PS(D(r,n-r))~\text{with  equality holding if and only if}~G\cong D(r,n-r)$.\\
$(ii)$ If $n\geq5$, then $PS(D(r,n-r))< PS(D(r-1,n-r+1))$.
\end{lemma}

Checking $B_{3}(p,q,r)$ in $\mathscr{B}^{3}_{n}(p,q,r)$, we know that $p\geq 0$, $q\geq 0$ and $r\geq 0$, and at most one of them
is $0$. Without loss of generality, suppose that $p\geq 1$, $q\geq 1$ and $r\geq 0$ in $B_{3}(p,q,r)$. Thus
 $\mathscr{B}^{3}_{n}(p,q,r)=\mathscr{B}^{3}_{n}(p,q,0)\cup
 \{\mathscr{B}^{3}_{n}(p,q,r)-\mathscr{B}^{3}_{n}(p,q,0)\}$. Let $B_{3}'(1,1,0,n-4)$ be a graph obtained by joining a vertex of degree 2 in $B_{3}(1,1,0)$ and a pendant vertex of $P_{n-4}$, and let $B_{3}^{1}(1,1,0,n-4)$ be a graph obtained by joining a vertex of degree 3 in $B_{3}(1,1,0)$ and a pendant vertex of $P_{n-4}$.
\begin{lemma}\label{art53}
Let $G\in \mathscr{B}^{3}_{n}(p,q,0)$ be a bicyclic graph with $n$ vertices. If $n\geq 6$, then
\begin{eqnarray*}
PS(G)\leq 14F(n-3)+6F(n-4),
\end{eqnarray*}
where the equality holds if and only if $G\cong B_{3}'(1,1,0,n-4)$.
\end{lemma}
\begin{proof}

Let $G\in \mathscr{B}^{3}_{n}(p,q,0)$ be a bicyclic graph with $n\geq 6$ vertices. Without loss of generality, suppose $p\leq q$. Now  we consider  two cases as follows.

{\bf Case 1.}  Suppose that  $p+q+2\geq5$ in induced subgraph $B_{3}(p,q,0)$ of $G$.  For convenience, assume that $u$  and $v$ are the vertices of degree 3 in induced subgraph $B_{3}(p,q,0)$ of  $G$. By Lemma \ref{art21}, we have that $PS(G)=PS(G-uv)+PS(G-\{u,v\})+2PS(G-C_{p+2})+2PS(G-C_{q+2})$.
By Lemma \ref{new3}, we obtain that $PS(G-\{u,v\})\leq2F(n-3)$ with equality holding if and only if $G-\{u,v\}=P_{2}\cup P_{n-4}$. Since $p\leq q$ and $p+q+2\geq5$, we have $p\geq1$ and $q\geq2$.
 By Lemma \ref{newwul2}, we  get that $PS(G-C_{p+2})\leq F(n-2)$ and $PS(G-C_{q+2})\leq F(n-3)$.
Thus, we get that
\begin{eqnarray}\label{S1}
PS(G)\leq PS(G-uv)+4F(n-3)+2F(n-2).
\end{eqnarray}

Assume that $G-uv=C_{p+q+2}$. By (\ref{S1}),  we have $PS(G)\leq F(n+1)+3F(n-1)+2F(n-3)+2=13F(n-3)+6F(n-4)+2$. If $n=6$, then $G$ is $B_{3}(2,2,0)$ or $B_{3}(1,3,0)$.
By Lemma \ref{art21}, we obtain  that $PS(B_{3}(2,2,0))=32$, $PS(B_{3}(1,3,0))=31$ and $PS(B_{3}^{'}(1,1,0,2))=34$. So, $PS(B_{3}^{'}(1,1,0,2))>PS(G)$. If
 $n\geq7$, $PS(B_{3}^{'}(1,1,0,n-4))-PS(G)\geq[14F(n-3)+6F(n-4)]-[13F(n-3)+6F(n-4)+2]=F(n-3)-2\geq~F(4)-2=1>0$. So,~$PS(B_{3}^{'}(1,1,0,n-4))>PS(G)$.

Suppose that $G-uv$ is a unicyclic graph with $n$ vertices excluding cycle $C_{p+q+2}$.  By Lemma \ref{art51}, and $p+q+2\geq5$, we have that $PS(G-uv)\leq~13F(n-4)+5F(n-5)$~with equality holding if and only if $G-uv\cong~D(5,n-5)$. From (\ref{S1}), we have  $PS(G)\leq13F(n-4)+5F(n-5)+4F(n-3)+2F(n-2)=11F(n-3)+10F(n-4)< 14F(n-3)+6F(n-4)$.

{\bf Case 2.} Assume that  $p+q+2=4$ in induced subgraph $B_{3}(p,q,0)$ of $G$. This implies that $p=q=1$.  Set that $G$ is obtained by identifying a vertex of $B_{3}(p,q,0)$ and a vertex of tree. By theorem \ref{art23}, $PS(G)\leq  PS(B_{3}^{'}(1,1,0,n-4))$ or $PS(G)\leq PS(B_{3}^{1}(1,1,0,n-4))$.
By Lemma \ref{art21}, we have  $PS(B_{3}^{'}(1,1,0,n-4))-PS(B_{3}^{1}(1,1,0,n-4))=3F(n-4)>0$. So,
$PS(G)\leq PS(B_{3}^{'}(1,1,0,n-4))=14F(n-3)+6F(n-4)$.  Set that $G$ is obtained by identifying some trees to at least two vertices of $B_{3}(p,q,0)$.
 There exists an edge $e(e\neq uv)$ of $B_{3}(1,1,0)$ in $G$ which has at least an endpoint   attached a tree.
By Lemma \ref{art21}, we obtain that $PS(G)=PS(G-e)+PS(G-V(e))+2PS(G-C_{3})+2PS(G-C_{4})$.
 By Lemma \ref{new3}, we have
$PS(G-V(e))\leq2PS(P_{n-4})$, $PS(G-C_{3})\leq2PS(P_{n-5})$ and $PS(G-C_{4})\leq2PS(P_{n-6})$.
By the structure of $G$, $G-e\neq D(3,n-3)$. By Lemma \ref{newwl}, we get $PS(G-e)<PS(D(3,n-3))$. So,
$PS(G)=PS(G-e)+PS(G-V(e))+2PS(G-C_{3})+2PS(G-C_{4})$
$<PS(D(3,n-3))+2PS(P_{n-4})+4PS(P_{n-5})+4PS(P_{n-6})$
$=PS(D(3,n-3))+6PS(P_{n-4})=PS(B_{3}^{'}(1,1,0,n-4))$.
\end{proof}

\begin{lemma}\label{art54}
Let $G\in \mathscr{B}^{3}_{n}(p,q,r)-\mathscr{B}^{3}_{n}(p,q,0)$ be a bicyclic graph with $n$ vertices. If $n\geq 5$, then
\begin{eqnarray*}
PS(G)< PS(B_{3}'(1,1,0,n-4)).
\end{eqnarray*}
\end{lemma}

\begin{proof}
Let $u$ be  a vertex of degree 3 in induced subgraph  $B_{3}(p,q,r)$ of  $G$, and let
$u_{1}$ be a neighbor of $u$ in subgraph $P_{p}$ of $B_{3}(p,q,r)$.

Suppose that $n=5$. It is easy to know that $G\cong B_{3}(1,1,1)$. By Lemma \ref{art21}, we have $PS(B_{3}(1,1,1))=19$ and $PS(B_{3}^{'}(1,1,0,1))=20$. Thus, $PS(G)<PS(B_{3}^{'}(1,1,0,1))$.

Assume that $n=6$. This indicates that $G$ has three possible structures, i.e.,  $G\cong B_{3}(1,1,2)$,
 $G$ is isomorphic to $H$  obtained by joining a vertex of degree 2 in $B_{3}(1,1,1)$ to a single vertex, or $G$ is isomorphic to $M$ obtained by joining a vertex of degree 3 in $B_{3}(1,1,1)$ to a single vertex.
 By Lemma \ref{art21}, we obtain that $PS(B_{3}(1,1,2))=29$, $PS(H))=28$, $PS(M)=23$ and $PS(B_{3}^{'}(1,1,0,2))=34$.  These imply that
$PS(G)<PS(B_{3}^{'}(1,1,0,2))$.

Suppose that  $n\geq7$. By Lemma \ref{art21}, we get that $PS(G)=PS(G-uu_{1})
+PS(G
-\{u,u_{1}\})+2PS(G-C_{p+q+2})+2PS(G-C_{p+r+2})$.
By Lemma \ref{art51} and $p+r+2\geq4$, we have $PS(G-uu_{1})\leq  PS(D(4,n-4)$.
 Since $G
-\{u,u_{1}\}$, $G-C_{p+q+2}$ and $G-C_{p+r+2}$ are forests, by Lemma \ref{newwul2}, we have
 $PS(G
-\{u,u_{1}\})\leq PS(P_{n-2})$, $PS(G-C_{p+q+2})\leq  PS(P_{n-4})$ and $PS(G-C_{p+r+2})\leq PS(P_{n-4})$. So, $PS(B_{3}^{'}(1,1,0,n-4))-PS(G)
\geq[PS(D(4,n-4))+3PS(P_{n-3})+2PS(P_{n-4})]-[PS(D(4,n-4)+PS(P_{n-3})+5PS(P_{n-4})]
=2PS(P_{n-5})-PS(P_{n-4})=2F(n-4)-F(n-3)=F(n-6)>0$.
\end{proof}

Combining Lemmas \ref{art53} and \ref{art54}, the theorem \ref{art3} holds.

\noindent{\bf Data Availability}

No data were used to support this study.

\noindent{\bf Conflicts of Interest}

The authors declare that they have no conflicts of interest.

 \noindent{\bf Acknowledgements}

 This first author is supported by the National Natural Science Foundation of China (No. 12261071), the Natural Science  Foundation of Qinghai Province (2020-ZJ-920).

\end{document}